\documentclass[10pt,a4paper]{elsarticle}
\usepackage{amsmath}
\usepackage{amsfonts}
\usepackage{amssymb}
\usepackage{graphicx}
\usepackage{caption}
\usepackage{subfig}
\usepackage{tabularx}
\usepackage{color}
\usepackage{algorithmicx} 
\usepackage{algorithm}
\usepackage{amsthm}
\usepackage{float}
\usepackage[toc,page]{appendix}
\newtheorem{theorem}{Theorem}[section]
\newtheorem{lemma}[theorem]{Lemma}

\begin{document}

\begin{frontmatter}

\title{An iterative scaling function procedure for solving scalar non-linear hyperbolic balance laws}

\author[mymainaddress]{Gino I. Montecinos\corref{mycorrespondingauthor}}
\cortext[mycorrespondingauthor]{Corresponding author}
\ead{gino.montecinos@uaysen.cl}

\address[mymainaddress]{Department of Natural Sciences and Technology, Universidad de Ays\'en, Coyhaique, Chile}

\begin{abstract}


The scaling of the exact solution of a hyperbolic balance law generates a family of scaled problems in which the source term does not depend on the current solution. These problems are used to construct a sequence of solutions whose limiting function solves the original hyperbolic problem. Thus this gives rise to an iterative procedure. Its convergence is demonstrated both theoretically and analytically. The analytical demonstration is in terms of a local in time convergence and existence theorem in the $L^2$ framework for the class of problems in which the source term $s(q)$ is bounded, with $s(0) = 0$, is locally Lipschitz and belongs to $C^2(\mathbb{R}) \cap H^1 (\mathbb{R}) $. A convex flux function, which is usual for existence and uniqueness for conservation laws, is also needed. 
For the numerical demonstration, a set of model equations is solved, where a conservative finite volume method using a low-dissipation flux is implemented in the iteration stages.  The error against reference solutions is computed and compared with the accuracy of a conventional first order approach in order to assess the gaining  in accuracy of the present procedure.  Regarding the accuracy only a first order scheme is explored because the development of a useful procedure is of interest in this work, high-order accurate methods should increase the computational cost of the global procedure.  Numerical tests show that the present approach is a feasible method of solution.  


\end{abstract}

\begin{keyword}

Hyperbolic balance laws \sep Conservation laws with source terms \sep Finite volume schemes \sep Iterative procedures.

\end{keyword}

\end{frontmatter}

\section{Introduction}

Hyperbolic balance laws play a crucial role in describing several phenomena in several fields of research, even source terms can appear artificially in numerical solutions of conservation laws via relaxation approaches, see \cite{Jin:1995a, Kawashima:1987a, Bouchut:1999a, Zeng:1999a} and reference therein. So, independent on the nature of the source term, one of the main issues in balance laws is the existence and uniqueness of solutions and of course, the ability of obtaining them.

For the case of systems coming from the relaxation approach admiting an equilibrium state, that is a state where the source term vanishes, the existence and uniqueness of solutions relies on local solutions around the equilibrium state which can be extended via continuation arguments and entropy assumptions to prove global existence, see \cite{Kawashima:1988a,Chen:1994a,Yong:2001a,Yong:2004a}. 

Balance laws can be written in a quasilinear form through the use of the Jacobian or derivative of flux function in the scalar case. The local existence and uniqueness of bounded measure solutions of quasilinear and semilinear equations can be provided by following classical ODE theory based on fixed point arguments. The requirements are that the initial condition function has to be bounded, measurable and regular enough, flux functions has to be locally Lipschitz continuous and the source term has to be locally bounded, measurable and locally Lipschitz as well, \cite{Bressan:2000a}.  Since quasilinear equations can blow up in finite time, the global existence in the $L^\infty$ and $L^{1}$ frameworks is done only in the case of semilinear equation. See Chapter 3 in \cite{Bressan:2000a} for further details. 

A different apprach was proposed in the pioneering work of Liu \cite{Liu1987}. Here the author has studied some systems of hyperbolic problems where the source term is in resonance with the advective part. By involving the bounded variation and monotone source terms the author has argued through the theory developed for hyperbolic conservation laws \cite{Liu:1981b}, that this class of hyperbolic balance laws have a global unique solution which is bounded also. In subsequent papers \cite{Dias:2002a, Tsuge:2017a} it has been proved the global existence in the $L^\infty$ and $L^{1}$ framework, the solution has been constructed by using entropy-admissible solutions obtained from parabolic equations in which the viscous coefficient goes to zero, it is known as the dimising viscosity method, also has been a key ingredient the compensated compactness of Tartar, \cite{Tartar:1979a}.  In the case of $L^\infty$, the methodology relies on the Kruzkov theory \cite{Kruzkov:1970a} to prove the uniqueness of entropy weak solutions.

Since a constructive proof of existence and uniqueness can lead to a procedure able to be implemented numerically. In this paper, a new local existence in time is presented which results into a feasible strategy for solving scalar hyperbolic balance laws. This is based on the so called {\it iteration algorithm} strategy presented in \cite{Cheng:2000a, CHEN:2001a,Chen:2010a} for solving semilinear elliptic partial differential equations with non-linear source term.  Iterative process is not a novel strategy of demonstration, this techniques has already been implemented for the existence of global solution of hyperbolic conservation laws, \cite{Douglis:1952, Friedrichs:1948a}, where the iteration is carried out in order to linearize the convective terms around a solution obtained in the previous iteration. These methods require flux functions to be uniform Lipschitz with respect to each argument. Furthermore, the initial conditions and the solutions have to be smooth, at least has to contain a continuous first derivative. Since, in hyperbolic equations singularities can appear in finite time, \cite{Dafermos:2016}.  A feasible method has to be able to deal with discontinuous solutions. That is, the method needs to incorporate weak solutions. Here, a new iterative approach to prove local existence and uniqueness in the $L^2$ framework, hable to incoirpates these features, is presented. The prove is for entropy-satisfying solutions, through the vanishing viscosity and the compensated compactness, as carried out in \cite{Dias:2002a}. The demonstration will be based on an iterative process and thus the existence and uniqueness of solutions will ensure the convergence of the procedure for a class of problems in which the source term $s(q)$ is locally bounded, $s(0) = 0$, is locally Lipschitz and belongs to $C^2(\mathbb{R}) \cap H^1 (\mathbb{R}) $.  For the flux function usual requirements as in \cite{Dias:2002a, Tsuge:2017a} are needed. The approach presented in this work, is more closed to that in \cite{Cheng:2000a, CHEN:2001a,Chen:2010a}, in the sense that a sequence of auxiliary problems are constructed by scaling the exact solution and the original equations as well, the main feature of these auxiliary problems is that these contain source terms which do not depend on the current state that means these are decoupled from the state. Then, a convergent sequence of solutions to these auxiliary problems is obtained, where the limiting function is a weak solution of the balance law.

This constructive demonstration can be translated into a numerical procedure in which a scheme for solving hyperbolic balance laws where the source terms do not depend on the sate, is involved. The procedure does not depend on a particular method but some minimal requirements are needed, particularly, the numerical solution must be an entropy satisfying one. Thus conservative scheme in the finite volume framework are suitable methods to be explored.

 For a practical implementation, a space-time mesh common to every auxiliary problem is considered, since the numerical stability in finite volume schemes can depend on the solution, we have to choose a small enough CFL coefficient such that every auxiliary problem is solved with the same stable scheme. In this work, we profit from  the recent low-dissipation scheme  \cite{Toro:2020a}, called by the authors FORCE$-\alpha$ which is suitable for problems requiring small CFL coefficients as needed here.
 
 
To show that the procedure is computationally feasible, the error against reference solutions is computed. Furthermore, a comparison with the accuracy of a conventional first order approach is also carried out to assess the gaining in accuracy of the present approach.

This work, is organized as follows. In the section \ref{sec:scaling-functions}, the formulation of the problem is presented. In the section \ref{sec:convergence}, the existence and uniqueness of solutions and consequently the convergence of the procedure are proved. In the section \ref{sec:experiments}, numerical experiments are shown to illustrate the applicability of the present approach. In section \ref{sec:experiments}, the conclusions and remarks are carried out. 

\section{The scaling function procedure}\label{sec:scaling-functions}
Let us consider the following one dimensional partial differential equation
\begin{eqnarray}
\label{eq:1} 
\begin{array}{c}
\partial_t u + \partial_x f ( u) = s( u  ) \;, x\in [a,b]  \;, t \in (0,T]\\

   u(x,0)                           = u_0(x) \;,

\end{array}
\end{eqnarray}
where $u_0$ is a prescribed function, $f(u)$ and $s(u)$ are the so called flux function and source function, respectively. Let us assume that an exact solution $u$ of (\ref{eq:1}) is available in $[0,T]$, for some $T<\infty$. Then, for any sequence of positive numbers $\beta_0, \beta_1,...$ such that the functions $v^n =
 \frac{ u}{ \beta_n}$ are a scaling of the solution of (\ref{eq:1}). This scaling leads to the scaling of the original equation which has the form
\begin{eqnarray}
\label{eq:system-for-n+1}
\begin{array}{c}
\frac{ \partial }{\partial t}(v^{n+1}) + \frac{ \partial}{\partial x}( \frac{ f( \beta_{n+1} \cdot v^{n+1} )}{\beta_{n+1} }) = \frac{ s(\beta_{n} v^n)}{\beta_{n+1}}  \;,
\\
v^{n+1}(x,0) = \frac{ u_0(x)}{ \beta_{n+1}} \;,
\end{array}
\end{eqnarray} 
and it is referred to us the {\it scaled problem}.

In this work, we are going to be interested on the converse, that is, to identify the conditions which guarantee, that if (\ref{eq:system-for-n+1}) has a solution $v^n$, for a given positive constant $\beta_{n}$, it is possible to find convergent subsequences $\{ v^n\}$ and $\{ \beta_n\}$ which converge to $v^\infty$ and $\beta_\infty$, respectively and $q = \beta_\infty \cdot v^{\infty}$ is a solution of (\ref{eq:1}).

Therefore, the hyperbolic balance laws (\ref{eq:1}) can be solved through the following iterative process, in which a sequence of solutions to auxiliary problems as (\ref{eq:system-for-n+1}) is constructed.

\begin{itemize}

\item {\bf Step 1:} Provide an arbitrary $v^0$ and set  $\beta_0 = 1$. Alternatively, we can take $\beta_0 = \max\{||v^0|| \}$.

\item {\bf Step 2:} Given $v^n $ and $\beta_n$, do solve (\ref{eq:system-for-n+1}). Since the problem depends on $\beta_{n+1}$ which is still unknown, we cannot solve it directly.   However, in the sense of distributions, we can reformulate the problem (\ref{eq:system-for-n+1}) as: Given $v^n$ and $\beta_n$, do solve 
\begin{eqnarray}
\label{eq:aux-equation}
\begin{array}{c}
\partial_t w +\partial_x f(w) = s(\beta_n v^n) \;, t \in [0,T] \;, \\
w(x,0) = u_0(x)\;.
\end{array}
\end{eqnarray}
Then, by using this solution, we define $\beta_{n+1} = \frac{1}{|| w || }$ and $v^{n+1} = w \cdot || w ||$. Here, $|| w ||$ is a suitable norm in space and time.

\item {\bf Step 3:}  If $E_n \leq Tol$ then stop, for some given tolerance $Tol$, where $E_n := |\beta_{n} - \beta_{n+1}|$. Otherwise, go to {\bf Step 2}

\end{itemize}

In the following section we are going to provide the conditions and theoretical results which guarantee the convergence of the procedure for solving hyperbolic balance laws through scaling functions.

\section{Convergence in the $L^2$ framework for the scaled function procedure}\label{sec:convergence}

In this section we are going to present the conditions and the corresponding results of existence and uniqueness  of weak solutions, in the $L^2(\mathbb{R})$ framework, for the procedure in section \ref{sec:scaling-functions}  given by {\bf Step 1}, {\bf Step 2} and {\bf Step 3}. Since balance laws with source terms can blow-up in finite time $t_b$, \cite{Bressan:2000a}, we are going to be interested on solutions up to $T \ll t_b$.

Notice that the equation (\ref{eq:aux-equation}) has the general form
\begin{eqnarray}
\label{eq:model-scaling}
\begin{array}{c}
\frac{\partial u}{\partial t} + \frac{\partial f(u)}{\partial x} = \tilde{s}(x,t) \;, t \in [0,T], x \in \mathbb{R } \;, \\
u(x,0) = u_0(x)\;,  
\end{array}
\end{eqnarray}
where $ \tilde{s}(x,t) = s(\beta_{n} \cdot v^n) $.  Notice that in the {\bf Step 2}, in section \ref{sec:scaling-functions} a sequence of functions $\{v^n \}$ and real numbers $\{ \beta_n\}$ can be generated if the solution to every problem (\ref{eq:model-scaling}) there exists. So, the first task is to prove the existence and uniqueness for these problems.

To apply existent theories the flux function is assumed to satisfy the following:
\begin{itemize}
\item The flux is genuinely non-linear
\begin{eqnarray}
\label{eq:prop-f-1}
f\in C^2 (\mathbf{R}) \;, f''(u)>0 \;, f(0) = f'(0) = 0 \;.
\end{eqnarray}
\item This has the following behaviour at infinity
\begin{eqnarray}
\label{eq:prop-f-2}
|f|\leq C (1 + |u|), |u| \rightarrow \infty \;.
\end{eqnarray}
\end{itemize}

Notice that, in the case of quasilinear equations 
\begin{eqnarray}
\begin{array}{c}
\partial_t q + \lambda \partial_x q = \tilde{s}(x,t) \;,
\end{array}
\end{eqnarray}
global existence and uniqueness is proved in the case in which $\lambda$ is a constant, see \cite{Bressan:2000a} for further details.

Before providing the main results, given constant values $\xi^-,\xi^+$ with $\xi^- < \xi^+$ and $R>0$, let us introduce the constant values $C^\xi_m$ and $C^\xi_p$, defined by 
\begin{eqnarray}
\begin{array}{l}
C^{\xi}_m = \min_{y \in [\xi^-, \xi^+]} \{  \lambda(u_0(y)) \} \;, \\
C^{\xi}_p = \max_{y \in [\xi^-, \xi^+]} \{  \lambda(u_0(y) + R \cdot T) \} \;,
\end{array}
\end{eqnarray}
where $f'(u) = \lambda(u)$. Furthermore, let us define the region $ \tilde{D}_{R}(\xi^-,\xi^+) \subset [0, \infty)  \times \mathbb{R}$, as the set containing the pairs $(\xi, \tau)$ such that
$ \gamma(t), \gamma_s(t) \in  [ \xi^- + t \cdot C_m^\xi, \xi^+ + t \cdot C^\xi_p] \subset \tilde{D}_{R}(\xi^-,\xi^+)  $ for all $t \in [0, \tau]$, where $\gamma(t)$ and $\gamma_s(t)$
are the right lines  given by $ \gamma(t) = \xi + t \lambda(u_0(\xi))$ and $ \gamma_s(t) = \xi + t \lambda(u_0(\xi) + R \cdot T)$.

\begin{lemma}\label{theo-1}
If $\tilde{s}: \mathbb{R} \times [0,T] \rightarrow \mathbb{R}$ satisfies:
\begin{itemize}
\item There exists $K_s$ such that $ | \tilde{s}(x,t) | \leq K_s $.

\item There exist constant values $ \xi^-, \xi^+$, with $\xi^- < \xi^+ $, such that $\tilde{s}$  has a compact support $\omega_s \subseteq \tilde{D}_{K_s}(\xi^-, \xi^+ ) $.
\end{itemize}

If $u_0(x) \in L^{2}(\mathbb{R})$. Then the problem (\ref{eq:model-scaling}) has an exact solution $\bar{u} \in L^2((0,T),\mathbb{R})$, which is also bounded in $ \tilde{D}_{K_s}(\xi^-,\xi^+)$.

\end{lemma}
\begin{proof}

Let us consider a sequence  $\{ u^\varepsilon \} \in L^2(\mathbb{R})$ of (smooth) solutions of the following parabolic equation 
\begin{eqnarray}
\begin{array}{c}
\frac{\partial u^\varepsilon}{\partial t} + \frac{\partial f (u^\varepsilon)}{\partial x} = \tilde{s}(x,y ) + \varepsilon \frac{\partial^2 u^\varepsilon}{\partial x^2} \;, t\in[0,T] \;,
\\
u^\varepsilon(x,0) = u_0^\varepsilon(x) \;,
\end{array}
\end{eqnarray}
where $  u_0^\varepsilon(x) $ is a converging sequence of smooth functions to $ u_0(x) $. 
Notice that this function has the following entropy condition
\begin{eqnarray}
\begin{array}{c}
   \frac{\partial  }{\partial t} \biggl( \frac{ (u^\varepsilon)^2}{2} \biggr)
 + u^\varepsilon \cdot \frac{\partial  f(u^\varepsilon)}{\partial x}
= u^\varepsilon \cdot \tilde{s}(x,y ) + \varepsilon \frac{\partial^2 \phi( u^\varepsilon) }{\partial x^2} - \varepsilon \phi'' [\frac{ \partial u^\varepsilon}{\partial x}]^2  \;.

\end{array}
\end{eqnarray}
On the other hand, since $\omega_s \subseteq \tilde{D}_{R}(\xi^-,\xi^+) $, by following the characteristic curves, we can obtain
\begin{eqnarray}
\begin{array}{c}
\displaystyle
\frac{d}{dt}
\int_\mathbb{R} \frac{ (u^\varepsilon )^2}{2} dx + \varepsilon \int_\mathbb{R} ( \frac{ \partial u^\varepsilon}{\partial x})^2 dx
=
\int_\mathbb{R} u^\varepsilon \tilde{s}dx 
\end{array}
\end{eqnarray}
\begin{eqnarray}
\begin{array}{c}
\displaystyle
\frac{d}{dt}
   \int_\mathbb{R} \frac{ (u^\varepsilon )^2}{2} dx
 + \varepsilon \int_\mathbb{R} ( \frac{ \partial u^\varepsilon}{\partial x})^2 dx

 + \int_\mathbb{R} |u^\varepsilon| |\tilde{s}(x,t) |dx 
 
\leq 
\\
\displaystyle

   \int_\mathbb{R} |u^\varepsilon| |\tilde{s}(x,t) | dx  
 + \int_\mathbb{R} |u^\varepsilon| |\tilde{s}(x,t)  |dx 

\leq
\\ 
  \displaystyle

   2 \int_\mathbb{R} |u^\varepsilon| \cdot  |\tilde{s}(x,t) |   dx \;.

\end{array}
\end{eqnarray}
By the H\"older inequality we obtain
\begin{eqnarray}
\begin{array}{c}
\displaystyle

  \int_\mathbb{R} |u^\varepsilon| \cdot  |\tilde{s}(x,t) | dx 

\leq
\displaystyle
  \int_\mathbb{R} \frac{|u^\varepsilon|^2}{2} dx +  M \cdot \frac{K_s^2}{2} \;,

\end{array}
\end{eqnarray}
with $ M = \int_{\tilde{D}_{K_s}(\xi^-, \xi^+)} dx = \frac{ T}{2} ( 2( \xi^+ - \xi^-) + T(C_p^\xi - C_m^\xi) )$.
Then by Gronwall's inequality
\begin{eqnarray}
\label{eq:bound-1-teo-1}
\begin{array}{c}
\displaystyle
   \int_\mathbb{R} \frac{ (u^\varepsilon )^2}{2} dx
 + \varepsilon \int_0^t\int_\mathbb{R} \biggl( \frac{ \partial u^\varepsilon}{\partial x} \biggr)^2 dx ds

 + \int_0^t \int_\mathbb{R} |u^\varepsilon| |\tilde{s}(x,t) |dx ds 
 
\leq  

K(T) \;,
\end{array}
\end{eqnarray}
in $[0,T]$. This establishes a uniform $L^2$ bound for the sequence $\{u^\varepsilon \}$.  On the other hand, any entropy $\phi \in C^2(\mathbb{R}) $ having a compact support $\omega$, for the equation (\ref{eq:model-scaling}) satisfies 
\begin{eqnarray}
\label{eq:gen-entropy}
\begin{array}{c}

\partial_t \phi( u^\varepsilon) + \partial_x \psi( u^\varepsilon) 
= 
 \phi' \tilde{s} + \varepsilon \partial_{xx} (\phi( u^\varepsilon)) 
- \varepsilon \phi''(u^\varepsilon) \cdot (\partial_x u^\varepsilon )^2 \;,    

\end{array}
\end{eqnarray}
where $\psi' = \phi' \cdot f'$.  Since $\phi$ has a compact support, then $ \phi'$ is bounden within  $\omega$, let say by a constant $a$. So
\begin{eqnarray}
\label{eq:bound-2-teo-1}
\begin{array}{c}
\displaystyle
\int_0^T \int_\mathbb{R} 
| 
\phi'(u^\varepsilon) \cdot \tilde{s}(x,t) 
|
dx dt
\leq 
 a ||\tilde{s}||_1 = a M \cdot K_s
 \;.

\end{array}
\end{eqnarray}
Therefore, by combining (\ref{eq:bound-1-teo-1}), (\ref{eq:gen-entropy}) and (\ref{eq:bound-2-teo-1}) we note that $\partial_t \phi( u^\varepsilon) + \partial_x \psi( u^\varepsilon) $ is bounded in  $\omega$, and then this lies in a compact set of $H^{-1}_{loc}( [0,T] \times \mathbb{R})$. Therefore, from the Theorem 3.2 and Corollary 3.2 in \cite{Schonbek:1982a}, there exists a subsequence of $ \{ u^\varepsilon \}$ still denoted by $\{ u^\varepsilon \}$, which converges weak to some $\bar{u}$, furthermore,  since $f$ it is strictly convex,  from Theorem 2.1 in \cite{Dias:2002a}, $u^\varepsilon$ converges strong to $\bar{u}$. Furthermore, $ f(u^\varepsilon) $ converges weak to $f(\bar{u})$, that is
\begin{eqnarray}
\begin{array}{c}
\displaystyle
\int_0^T \int_\mathbb{R} f(u^\varepsilon) \phi dxdt \rightarrow \int_0^T \int_\mathbb{R} f( \bar{u}) \phi dxdt,
\end{array}
\end{eqnarray}
for all $ \phi \in C^\infty_c([0,T] \times \mathbb{R} ) $. Since the source term, does no depend on the state $u^\varepsilon$, using a standard diagonalization procedure, \cite{Attouch:2014a}, the result holds.
Furthermore, from (\ref{eq:bound-1-teo-1}) the solution is bounded, in $L^2$, for each $t \in [0,T]$ and thus the solution is also bounded in $ \tilde{D}_{K_s}(\xi^-,\xi^+)$.

\end{proof}

Since the source term is not coupled with the conservation law, this may generate degeneracy in the sense that infinitely many asymptotic states may be possible when $t\rightarrow \infty$ \cite{Dafermos:2015a}. However, in the local in time case, a unique solution exists if the source term has a support which lies in $\tilde{D}_{K_s}(\xi^-,\xi^+)$.  As we shall see later, it is enough to guarantee, in particular cases, the existence of the solution to the original equation (\ref{eq:1}).

The next step is to prove conditions on $s(u)$ such that a subsequence of $ \{ w^n\}$ still called $ \{ w^n\}$, with $w^n =v^n \cdot \beta_n $ where $v^n$ is the exact solution to (\ref{eq:system-for-n+1}), is convergent to an exact solution of (\ref{eq:1}). Indeed, we assume the following:
\begin{itemize}

\item {\bf H1}: $s : \mathbb{R} \rightarrow \mathbb{R}$ is locally Lipschitz continuous and $s(0) = 0$.    

\item {\bf H2}: $s \in C^{2}(\mathbb{R}) \cap H^1(\mathbb{R})$.
\end{itemize}
\begin{theorem}\label{theo-2} 

Let $f:\mathbb{R} \rightarrow \mathbb{R}$ be a function satisfying (\ref{eq:prop-f-1}) and (\ref{eq:prop-f-2}) and  let $ s : \mathbb{R} \rightarrow \mathbb{R}$  be a function satisfying {\bf H1} and {\bf H2}. If $supp(u_0)$ is bounded. Then the sequence, $\{ w^{n+1}\}$ with  $w^{n+1} = \beta_{n+1} \cdot v^{n+1}$, where $v^{n+1}$ is a solution of (\ref{eq:system-for-n+1}) converges to  a weak solution $\bar{w} \in L^1_{loc}$ of (\ref{eq:1}).

\end{theorem}

\begin{proof}

Let us define $  supp(u_0) = [\xi^-, \xi^+]$. From  hypothesis {\bf H1} and {\bf H2}, in any open set $\Omega$ such that $0 \in \Omega$ there exist constant values $ L_s$ and $C_s$ such that 
$$ | s(u_1) - s(u_2) | \leq L |u_1 - u_2 |$$ and $$ |s(u_1)| < C_s \;, |s'(u_1) | < C_s \;,$$
for each $u_1, u_2\in \Omega$.  

Therefore if $w^{0}(x,t) $ is any bounded function in $ (x,t)\in \tilde{D}_{C_s}(\xi^-, \xi^+)$ with $w^{0}(x,t) \in \Omega$. Then $  supp( s(w^0) ) \subseteq \tilde{D}_{C_s}(\xi^-, \xi^+)$. So, the conditions of Lemma \ref{theo-1} are satisfied.  So, a solution $w^1$ there exists and this is bounded in $ \tilde{D}_{C_s}(\xi^-, \xi^+)$.

By the same arguments,  we  obtain a subsequence, still called here $ \{ w^n \}$  which in virtue of Lemma \ref{theo-1} solves the equation
\begin{eqnarray}
\label{eq:model-scaling-aux}
\begin{array}{c}
\frac{\partial w^n}{\partial t} + \frac{\partial f(w^n)}{\partial x} = s( w^{n-1}) \;, t \in [0,T], x \in \mathbb{R } \;, \\
u^{n}(x,0) = u_0(x)\;
\end{array}
\end{eqnarray}
and $supp(s( w^{n-1}(x,t))) \subseteq \tilde{D}_{C_s}(\xi^-, \xi^+) $.  So, on each curve $x(t)$ defined by $ \frac{ dx}{dt} = \lambda( w^n(x,t))$ the function $w^n(x,t)$ satisfies
\begin{eqnarray}
\begin{array}{c}
\frac{d w^n(x,t)}{dt} = s(w^{n-1}(x,t)) \;,
\end{array}
\end{eqnarray}
thus, we obtain
\begin{eqnarray}
\label{eq:model-scaling-aux-1}
\begin{array}{c}
\frac{ d}{dt} \biggl[  \frac{ ( w^n(x,t) - w^{m}(x,t) )^2 }{2}\biggr] = ( w^n(x,t) - w^{m}(x,t) ) \cdot \biggl(

\partial_t (w^n(x,t) - w^m(x,t)) 
\\ + \frac{dx}{dt} \cdot \partial_x (w^n(x,t) - w^m(x,t) ) 
\biggr)
\;.
\end{array}
\end{eqnarray}
After some manipulations we obtain
\begin{eqnarray}
\label{eq:model-scaling-aux-2}
\begin{array}{c}

\frac{ d}{dt} \biggl[  \frac{ ( w^n(x,t) - w^{m}(x,t) )^2 }{2}\biggr] =
\biggl( w^n(x,t) - w^{m}(x,t) \biggr) \cdot 
\biggl(
s(w^{n-1}) - s(w^{m-1}) \\
 + ( \frac{ dx}{dt} - \lambda(w^n) ) \cdot \frac{\partial w^n}{\partial x} 
 - ( \frac{ dx}{dt} - \lambda(w^m) ) \cdot \frac{\partial w^m}{\partial x}
\biggr)
\;.
\end{array}
\end{eqnarray}
Since $s \in C^2( \mathbb{R})$ and $ f\in C^2(\mathbb{R})$, we assume that given $x$ and $t$, there exist $\theta_s(x,t)$ and $\theta_f(x,t)$ such that 
\begin{eqnarray}
\label{eq:theta-s}
\begin{array}{c}
\beta( \theta_s(x,t) ) \cdot ( w^{n-1}(x,t) - w^{m-1}(x,t) ) = s( w^{n-1}) - s(w^{m-1}) 
\end{array}
\end{eqnarray}
and
\begin{eqnarray}
\label{eq:theta-f}
\begin{array}{c}
\lambda(\theta_f(x,t)) \cdot ( w^{n}(x,t) - w^{m}(x,t) ) = f( w^{n}(x,t)) - f(w^{m}(x,t)) \;.
\end{array}
\end{eqnarray}
So, (\ref{eq:model-scaling-aux-2}) can be written as 
 \begin{eqnarray}
\label{eq:model-scaling-aux-3}
\begin{array}{c}
\frac{ d}{dt} \biggl[  \frac{ ( w^n - w^{m} )^2 }{2}\biggr] = 
\\
\frac{( w^n - w^{m} )^2}{2} \cdot \biggl\{

2 \beta(\theta_s)  \cdot \delta^{n,m} + \frac{ dx}{dt} \cdot \partial_x \ln\biggl( [ w^n - w^m]^2\biggr) \\
 - \lambda(\theta_f) \cdot \partial_x \ln\biggl( \lambda(\theta_f)^2 \cdot (w^n- w^m)^2 \biggr) 
\biggr\}
\;,
 \end{array}
\end{eqnarray}
where $\delta^{n,m} = \frac{ w^{n-1}-w^{m-1}}{w^n- w^m}$, $\beta = \frac{d s(u)}{dt}$ and $\lambda (u) = \frac{ df(u)}{du}$.  This yields
\begin{eqnarray}
\label{eq:model-scaling-aux-4}
\begin{array}{c}
\displaystyle
\frac{ d}{dt} 

\biggl[  \frac{ ( w^n - w^{m} )^2 }{2}
\cdot
exp\biggl( - \displaystyle \int_0^t \biggl\{
2 \beta(\theta_s)  \cdot \delta^{n,m} + \frac{ dx(\tau)}{dt} \cdot \partial_x \ln\biggl( [ w^n - w^m]^2\biggr) \\
 - \lambda(\theta_f) \cdot \partial_x \ln\biggl( \lambda(\theta_f)^2 \cdot (w^n- w^m)^2 \biggr) 
\biggr\} d \tau 
\biggr)  
\biggr]
= 0
\;.
 \end{array}
\end{eqnarray}
Thus 
\begin{eqnarray}
\label{eq:model-scaling-aux-5}
\begin{array}{c}
\displaystyle

\biggl[  \frac{ ( w^n(x,t) - w^{m}(x,t) )^2 }{2}\biggr] \cdot exp \biggl( - \displaystyle \int_0^t \biggl\{

2 \beta(\theta_s(x,\tau))  \cdot \delta^{n,m}(x,\tau) \\

+ \frac{ dx(\tau)}{dt} \cdot \partial_x \ln\biggl( [ w^n(x,\tau) - w^m(x,\tau)]^2\biggr) \\
 - \lambda(\theta_f(x,\tau)) \cdot \partial_x \ln\biggl( \lambda(\theta_f(x,\tau))^2 \cdot (w^n(x,\tau)- w^m(x,\tau))^2 \biggr) 
\biggr\} d \tau \biggr)  
\\
= 
\displaystyle
 \frac{ ( w^n(x,0) - w^{m}(x,0))^2 }{2}\;.

 \end{array}
\end{eqnarray}
By the sake of simplicity, let us introduce
\begin{eqnarray}
\begin{array}{c}
\Gamma(x, t) = 
- \displaystyle \int_0^t \biggl\{

2 \beta(\theta_s)  \cdot \delta^{n,m} 

\displaystyle
+ \frac{ dx(\tau)}{dt} \cdot \partial_x \ln\biggl( [ w^n - w^m]^2\biggr) 
\\
 - \lambda(\theta_f) \cdot \partial_x \ln\biggl( \lambda(\theta_f)^2 \cdot (w^n- w^m)^2 \biggr) 
\biggr\} d \tau \;.
\end{array}
\end{eqnarray}
Since $w^n$ and $w^m$ are the limiting solution ($\varepsilon \rightarrow 0 $) of problems as (\ref{eq:model-scaling}) with $\tilde{s}(x,t)$ given by $  s( w^{n-1}(x,t))$ and $ s( w^{n-1}(x,t))$, respectively, where,  each $ w^n$ satisfies the entropy condition $ \partial_t ( (w^n)^2 / 2) + \partial_x ( w^n f(w^n) ) = w^n s(w^{n-1})$, then we assume that $( w^n - w^m)^2$ is bounded.  This guarantees that for all $\psi \in  C^{2} (\mathbb{R})$ we obtain that
\begin{eqnarray}
\label{eq:model-scaling-aux-6}
\begin{array}{c}
\displaystyle \int_\mathbb{R } 
\biggl[  \frac{ ( w^n(x,t) - w^{m}(x,t) )^2 }{2}\biggr] \cdot \displaystyle e^{  - \Gamma(x,t)  } \cdot \psi(x) dx 

  =
\\
\displaystyle  
    \int_\mathbb{R } \frac{ ( w^n(x,0) - w^{m}(x,0))^2 }{2} \cdot \psi(x) dx\;,
 \end{array}
\end{eqnarray}
tends weak to zero.   Hence, we deduce that $ \{w^n\}$ is a weak Cauchy-sequence therefore it is a weak convergent sequence to $\bar{w}$.  By the same arguments in Lemma \ref{theo-1}, we have $ f(w^n) \rightarrow f( \bar{w} )$. It is remaining to prove that $s(w^n) \rightarrow s(\bar{w})$.

Since, $f$ is convex $w^n \rightarrow \bar{w}$ strong and from (\ref{eq:theta-s}) and (\ref{eq:theta-f}) we argue that $s(w^n(x,t))$ converges punctually to $s(\bar{w}(x,t))$. Furthermore, from (\ref{eq:bound-1-teo-1}) we deduce that for any measurable set $E$ with a finite measure $E$ and any disk $D(0, R) = \{ v : |v|<R\}$, we have that
\begin{eqnarray}
\begin{array}{c}
\displaystyle
\int_0^T \int_{E} | s( w^n(x,t))| dx dt \leq \int_0^T \int_{ E \cap D(0, R)} | s( w^n) | dx dt + \\
\int_0^T \int_{\mathbb{R} \setminus D(0, R)} \frac{ |w^n (x,t)|}{R} \cdot | s( w^n)| dx dt  \leq  T( mean(E) \cdot || s||_1  + \frac{K(T)}{R} )\;.
\end{array}
\end{eqnarray}
So, $\{s(w^n) \}$ is uniformly bounded in any measurable $E$ having a finite measure. So, from Vitali's theorem, ( for further information see Chapter 2 in \cite{Rudin:1987a}), $s(\bar{w})$ is integrable and
\begin{eqnarray}
\begin{array}{c}
\displaystyle
\int_0^T \int_{E} s(w^n(x,t))dxdt \rightarrow \int_0^T \int_{E} s(\bar{w}(x,t))dxdt  \;.
\end{array}
\end{eqnarray}
Thus, we deduce that the convergence is in $L^1_{loc}(\mathbb{R})$. Therefore
\begin{eqnarray}
\begin{array}{c}
\displaystyle
\int_0^T \int_{\mathbb{R}} s(w^n(x,t))\phi dxdt \rightarrow \int_0^T \int_{\mathbb{R}} s(\bar{w}(x,t))\phi dxdt  \;,
\end{array}
\end{eqnarray}
for all $\phi \in C^\infty_c(\mathbb{R})$ and so the result holds.

\end{proof}

Notice that the existence and uniqueness result in this section means also the convergence of the iterative process given by {\bf Step 1}, {\bf Step 2} and {\bf Step 3}, introduced in the section  \ref{sec:scaling-functions}. It is implicit in the demonstration that the solution in {\bf Step 2} has to be an entropy satisfying one.
In the next section we are going to implement this procedure on a set of well known balance laws, where the source terms satisfy the conditions of the theorem \ref{theo-2}.

\section{Numerical experiments}\label{sec:experiments}

Here, we are going to implement the iterative process described in the section \ref{sec:scaling-functions} by using the conventional one-step finite volume formula, in conserved form given by
\begin{eqnarray}
\label{eq:one-step-fv-formula}
\begin{array}{c}
q_i^{n+1} = q_i^{n} - \frac{ \Delta t}{\Delta x} \cdot ( f_{i+\frac{1}{2}} - f_{i-\frac{1}{2}} ) + \Delta t \cdot S_i \;,
\end{array}
\end{eqnarray}
where $q_{i}^n$ is the cell average of the solution $q(x,t)$ in the space-time interval $[t^n, t^{n+1}]\times [x_{i-\frac{1}{2}}, x_{i+\frac{1}{2}}]$, the expression $ f_{i+\frac{1}{2}} $ represents a numerical flux and $S_i$ is the source term. We use a CFL type condition to obtain the time step, so this is a first order method in both space and time. Here we only limit to first order. However, experiments not shown here have evidenced that the procedure can also be applied by using second order scheme. In general any high-order method can be implemented, but we need to take care about the applicability of the present approach.

Notice that, (\ref{eq:one-step-fv-formula}) can be implemented for both the original equation (\ref{eq:1}) and the scaled approach as well. Of course, to solve directly (\ref{eq:1}) we use $S_i = s(q_{i}^n)$ and denote this solution as $q^R$. For solving through the iterative approach and thus via the solution of (\ref{eq:aux-equation}) we use $ S_i = \tilde{s}(x_i, t^{n}) $ for given space-time dependent functions $\tilde{s} $. 

The reference solution, when applied, is obtained with a second-order MUSCL-HANCOCK scheme using a fine mesh (1000 cells). For numerical implementations we use a fixed number of cells and time steps as well. So, in order to guarantee stable schemes for any source term $\tilde{s}$, we use small CFL coefficients. Despite there is a large number of schemes able to solve source terms with large time steps, we insist into using a simple method because of we are interested on the ability of the present approach to generate approximations as simple as possible. 
Here, we implement the new low-dissipation centred scheme, \cite{Toro:2020a} named FORCE-$\alpha$, which works very well with low values of CFL coefficients without penalizing the suitable amount of numerical dissipation. The flux function has the form
\begin{eqnarray}
\begin{array}{c}
f_{i+\frac{1}{2}}^{\alpha} = \frac{1}{2} ( f_{i+\frac{1}{2}}^{LW,\alpha} + f_{i+\frac{1}{2}}^{LF, \alpha})\,,
\end{array}
\end{eqnarray}
where $f^{LF,\alpha}_{i+\frac{1}{2}}$ and $f^{LW, \alpha}_{i+\frac{1}{2}}$ are the numerical fluxes of Lax-Friedrich and Lax-Wendroff, respectively, which are given by

\begin{eqnarray}
\begin{array}{l}
\displaystyle
f^{LF,\alpha}_{i+\frac{1}{2}}:= \frac{1}{2}( f(q_{i+1}^n) + f(q_{i}^n) ) - \frac{ 1}{ 2} \frac{ \Delta t}{\alpha \Delta x} ( q_{i+1} - q_i)  \,,\\
\\
\displaystyle
f^{LW, \alpha}_{i+\frac{1}{2}} := f( q_{i+\frac{1}{2}}^{LW,\alpha})  \;, \\

\displaystyle
q_{i+\frac{1}{2}}^{LW,\alpha} := \frac{1}{2}( q_{i+1}^n + q_{i}^n ) - \frac{ 1}{ 2} \frac{ \alpha \Delta x }{\Delta t} (f(q_{i+1}^n) + f( q_{i}^n) ) \;,
\end{array}
\end{eqnarray}
here, $\alpha$ is a suitable constant value which is involved in the numerical dissipation of the scheme. The interested readers may consult \cite{Toro:2020a} for further details.
Regarding the implementation, we first chose the maximum CFL coefficient, $c_{max}$ for which the FORCE scheme ($\alpha = 1$), \cite{Toro:1999a}, applied to the original hyperbolic balance law (\ref{eq:1}) is stable and then compute the parameter $\alpha$ required for the scheme (Algorithm C in \cite{Toro:2020a}) and the maximum range of CFL coefficients for which the numerical method depicts a viscosity which is comparable to such of the Godunov scheme, the most accurate first order monotone scheme with the minimal numerical viscosity. Despite this scheme is originally proposed for conservation laws, numerical experiments show that it can be also applied in the context of hyperbolic balance laws.

In order to assess the performance of the present methodology we compute the error with respect to the reference solution $\bar{u}$, for both type of solutions, $q^R$ and that obtained at each stage $k$ of the iterative approach $w^k$, denoted here by $ Err^R = || \bar{u}(\cdot, T) - q^R(\cdot, T) ||_{L_1} $ and $Err^k = || \bar{u}(\cdot, T) - w^k(\cdot, T) ||_{L_1} $, respectively. To assess the gaining in the accuracy incurred by the present approach we are going to compute the gaining coefficient given by $ \tau^k = \frac{ Err^k}{ Err^R}$.
If procedure is a feasible one in terms of accuracy, we expect the coefficient to be closed to $1$.

In numerical implementations, $w^k$, is approximated on cells $[t^n,t^{n+1}] \times [x_{i-\frac{1}{3}}, x_{i+\frac{1}{2}}]$, so $||w^k||$ used in {\bf Step 2} of section \ref{sec:scaling-functions}, will be the maximum value that the function $|w^k(x,t)|$ reaches at each time step $t^n$ and at cell center $x_i = \frac{ x_{i-\frac{1}{2}}+ x_{i+\frac{1}{2} } }{2}$ of  $[x_{i-\frac{1}{3}}, x_{i+\frac{1}{2}}]$, since it is a first order approximation, this corresponds to the cell average of the data.  So, $\beta_n$ provides a measure of the solution in both space and time.

\section{The linear advection-reaction equation}
Let us consider the advection-reaction equation
\begin{eqnarray}
\begin{array}{c}
\partial_t q + \partial_x( \lambda q ) = r q\;, x \in [0, 1],\; t \in[0, T]\;, \\

q(x,0) = exp( - 100 (x- \frac{1}{2})^2)\;.
 
\end{array}
\end{eqnarray}
Notice that this source term satisfies that $s(q)=r q$ is locally bounded in $\mathbb{R}$ and $s(0) = 0$. The source terms is also locally Lipschitz continuous, furthermore, this belongs to the class of $C^\infty(\mathbb{R})$. We can assume that $u_0(x)$ has bounded support in $[0,1]$, we implement periodic boundary conditions and hence we can assume that $ supp(u_0)\subset supp(s)$  is bounded. Therefore, $s\in H^1(\mathbb{R})$. That is, conditions in the Theorem \ref{theo-2} are satisfied, hence the iterative procedure provides the solution of this balance law. Here, we use the model parameters $r = -10$, $\lambda = 1$. The implementation is carried out with, $CFL = 0.18$, $\alpha = 5.6$, $100$ cells and $Tol = 10^{-7}$. Figure \ref{figu:lin-adv-reac-comp}, shows the functions $w^0,$ $w^1$, $w^2$ and $w^4$ at the final time $T = 0.25$ generated by the iterative process beginning with $v^0 = 0$ and $\beta_0 = 1$. This depicts a clear tendency to the exact solution. Table \ref{table:linear-advection-reaction} shows the results for the convergence of the iterative process for the linear advection-reaction case. The second column, shows the sequence $\{ \beta_n \}$, which in some sense provides, a measure of the function $w^n$ in space and time. The third column shows the $L^1$ error at the output time  between each solution $w^n$ and the exact solution. The fourth column shows the gaining factor, $\tau^n$. We observe that the converged solution $w^{15}$ has the same performance as the conventional first order approximation. Here, the convergence is in terms of the tolerance $Tol$. Figure \ref{figu:lin-adv-reac-comp-FirstOrder}, shows the converged solution, $w^{15}$ (circles), the approximate  solution obtained by the first order approximation $q^R$ (squares) and the exact solution (continuous line). We see that $w^{15}$ and $q^R$ are identical. 
\begin{figure}
\begin{center}
\includegraphics[scale=0.5]{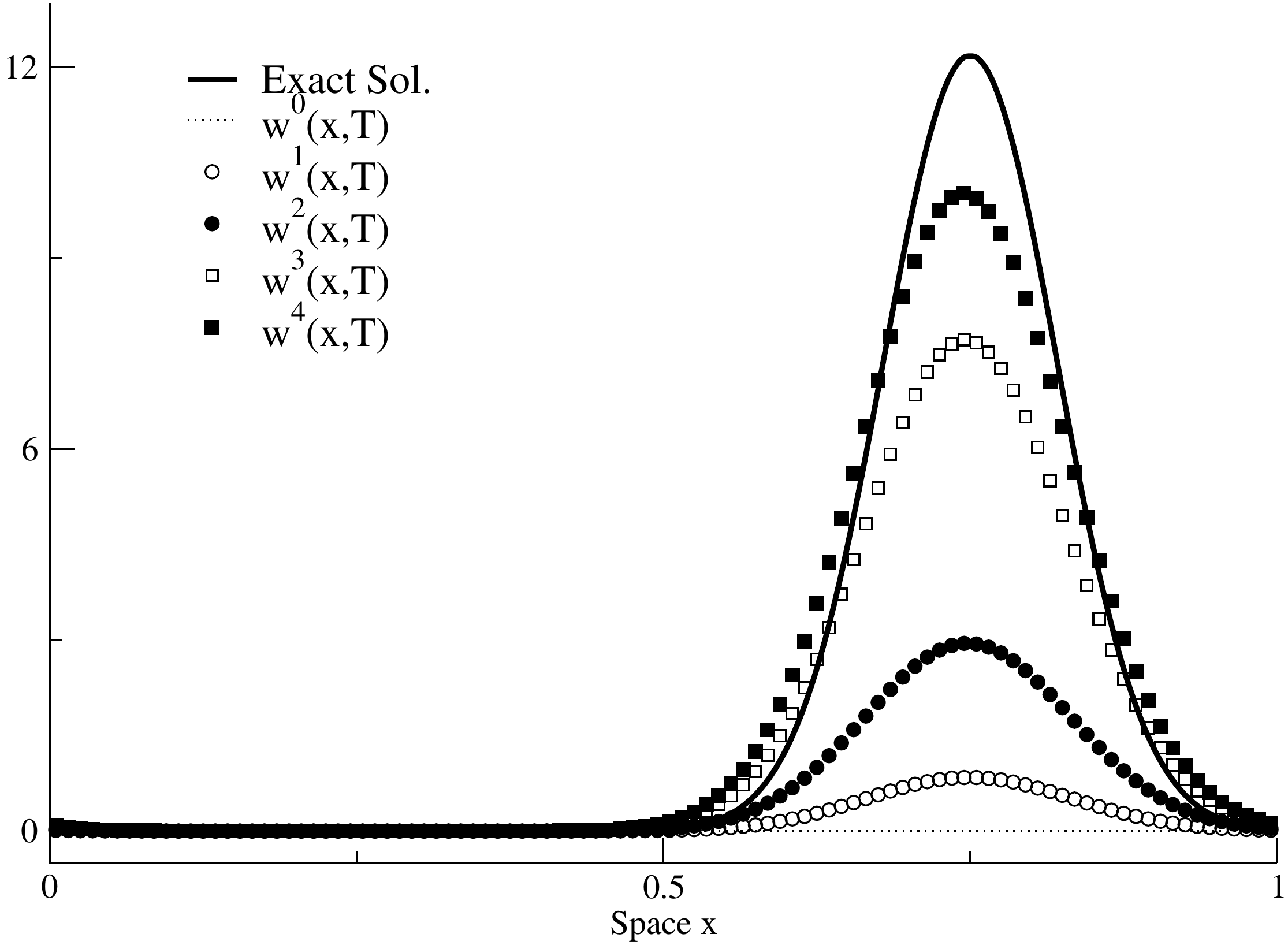}
\end{center}
\caption{Linear advection-reaction equation. Comparison of solutions $w^0$, $w^1$, $w^2$ and $w^4$ at $T =0.25$, against the exact solution for $\lambda=1$, $r = 10$, $100$ cells, $CFL = 0.18$ and $\alpha = 5.6$.  }\label{figu:lin-adv-reac-comp}
\end{figure}
\begin{figure}
\begin{center}
\includegraphics[scale=0.5]{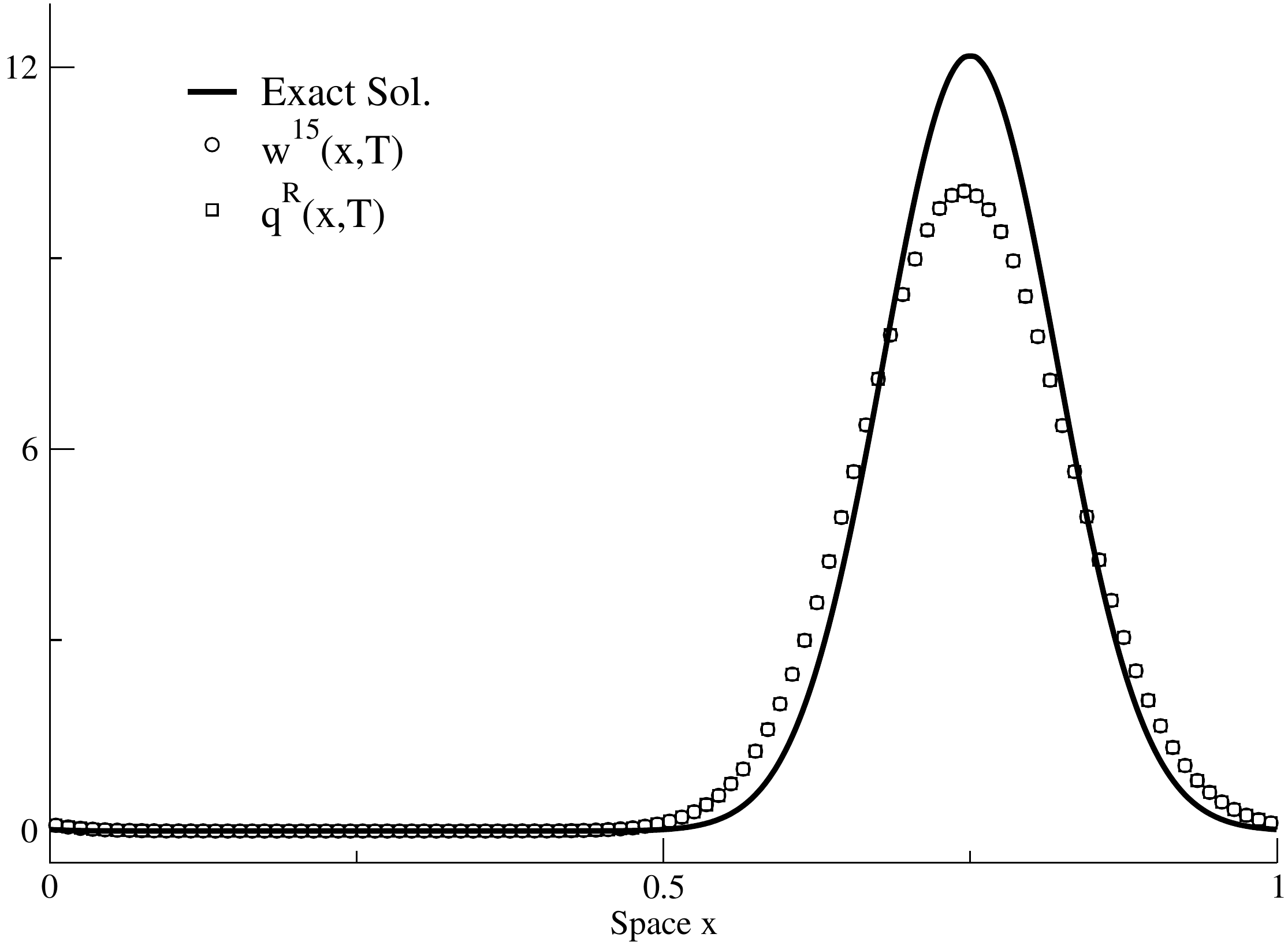}
\end{center}
\caption{Linear advection-reaction equation. Comparison of converged solution $w^{15}$, approximated firt order $q^R$ and the exact solution at $T =0.25$ for $\lambda=1$, $r = 10$, $100$ cells, $CFL = 0.18$, $\alpha = 5.6$ and $Tol = 10^{-7}$. }\label{figu:lin-adv-reac-comp-FirstOrder}
\end{figure}

\begin{table}
\begin{center}
\begin{tabular}{|c|c|c|c|}
\hline
$n$ & $\beta_n $ & $ Err^n$ & $\tau^n $\\
\hline
\hline

$  1 $ &  $ 1.002503184 $ &  $1.982060143 $ &  $ 0.192569 $ \\
$  2 $ &  $ 0.339221162 $ &  $1.541742664 $ &  $ 0.247567 $ \\
$  3 $ &  $ 0.179658402 $ &  $1.019974552 $ &  $ 0.374210 $ \\
$  4 $ &  $ 0.129557512 $ &  $0.649010090 $ &  $ 0.588102 $ \\
$  5 $ &  $ 0.110674919 $ &  $0.473993487 $ &  $ 0.805252 $ \\
$  6 $ &  $ 0.103359357 $ &  $0.410759385 $ &  $ 0.929216 $ \\
$  7 $ &  $ 0.100687214 $ &  $0.390106220 $ &  $ 0.978411 $ \\
$  8 $ &  $ 0.099806263 $ &  $0.384018730 $ &  $ 0.993921 $ \\
$  9 $ &  $ 0.099548143 $ &  $0.382245425 $ &  $ 0.998532 $ \\
$  10 $ &  $ 0.099480915 $ &  $0.381798718 $ &  $ 0.999700 $ \\
$  11 $ &  $ 0.099465245 $ &  $0.381707333 $ &  $ 0.999940 $ \\
$  12 $ &  $ 0.099461950 $ &  $0.381688495 $ &  $ 0.999989 $ \\
$  13 $ &  $ 0.099461320 $ &  $0.381684971 $ &  $ 0.999998 $ \\
$  14 $ &  $ 0.099461210 $ &  $0.381684368 $ &  $ 1.000000 $ \\
$  15 $ &  $ 0.099461192 $ &  $0.381684274 $ &  $ 1.000000 $ \\

\hline
\end{tabular}
\end{center}
\caption{Linear advection-reaction equation. Second column: parameter $\beta_n$. Third column: Error of $w^k$ with respect to the exact solution. Fourth column: gaining in accuracy in the $k$ stage. 
Parameters: $T =0.25$, $\lambda=1$, $r = 10$, $100$ cells, $CFL = 0.18$, $\alpha = 5.6$ and $Tol = 10^{-7}$.}\label{table:linear-advection-reaction}
\end{table}

\section{The Burger equation}
Let us consider the Burger equation with the non-linear source term
\begin{eqnarray}
\begin{array}{c}
\partial_t q + \partial_x( \frac{q^2}{  2} ) = q^4\;, x \in [0, 1],\; t \in[0, T]\;, \\

q(x,0) = \sin(2 \pi x)^4 \;,
 
\end{array}
\end{eqnarray}
endowed with periodic boundary conditions. Notice that this source as in the previous test, satisfies that $s(q)=q^4$ is locally bounded in $\mathbb{R}$ and $s(0) = 0$. The source terms is also locally Lipschitz continuous, furthermore, this belongs to $C^\infty(\mathbb{R})$, the class of infinitely continuous differentiable functions. We can assume that $u_0(x)$ has bounded support in $[0,1]$ and since periodic boundary conditions is implemented we have that $ supp(u_0)\subset supp(s)$  is bounded. Therefore, $s\in H^1(\mathbb{R})$. That is, this problem  does satisfy the conditions of the theorem \ref{theo-2}, hence the iterative procedure should provide the solution of this balance law. 

The implementation is carried out with, $CFL = 0.5$, $\alpha = 2.55$, $100$ cells and $Tol = 10^{-7}$. Figure \ref{fig:burger-u-to-power-4}, shows the functions $w^0,$ $w^1$, $w^2$ and $w^4$ at the final time $T = 0.12$ generated by the iterative process beginning with $v^0 = 0$ and $\beta_0 = 1$. 

Table \ref{table:burguer-u-to-power-4}, shows the results for the convergence of the iterative process. In the second column, it is shown the sequence $\{ \beta_n \}$. The third column shows the $L^1$ error at the output time between each solution $w^n$ and the reference solution. The fourth column shows the gaining factor, $\tau^n$. We observe that the converged solution $w^9$ has the same performance as the conventional first order approximation. Here, again the convergence is in terms of the tolerance $Tol$.  Figure \ref{fig:burger-u-to-power-4-comparison}, shows the converged solution, $w^9$ (circles) the approximate  solution obtained by the first order approximation $q^R$ (squares) and the exact solution (continuous line). We see that $w^{8}$ and $q^R$ are identical.
\begin{table}
\begin{center}
\begin{tabular}{|c|c|c|c|}

\hline
$n$ & $\beta_n $ & $ Err^n$ & $\tau^n $\\
\hline
\hline

 $ 1 $ &  $ 1.002089090 $ &  $0.049488757 $ &  $ 0.561212 $ \\
 $ 2 $ &  $ 0.951721597 $ &  $0.031257750 $ &  $ 0.888537 $ \\
 $ 3 $ &  $ 0.932017487 $ &  $0.028248110 $ &  $ 0.983204 $ \\
 $ 4 $ &  $ 0.928935042 $ &  $0.027821698 $ &  $ 0.998273 $ \\
 $ 5 $ &  $ 0.928586351 $ &  $0.027777618 $ &  $ 0.999858 $ \\
 $ 6 $ &  $ 0.928557279 $ &  $0.027773939 $ &  $ 0.999990 $ \\
 $ 7 $ &  $ 0.928555437 $ &  $0.027773699 $ &  $ 0.999999 $ \\
 $ 8 $ &  $ 0.928555346 $ &  $0.027773687 $ &  $ 0.999999 $ \\
 $ 9 $ &  $ 0.928555346 $ &  $0.027773687 $ &  $ 1 $ \\

\hline
\end{tabular}
\end{center}
\caption{Burger equation: Second column: parameter $\beta_n$. Third column: Error of $w^k$ with respect to the reference solution. Fourth column: gaining in accuracy in the $k$ stage. 
Parameters: $T =0.12$, $100$ cells, $CFL = 0.5$, $\alpha = 2.55$ and $Tol = 10^{-7}$.}\label{table:burguer-u-to-power-4}
\end{table}

\begin{figure}
\begin{center}
\includegraphics[scale=0.5]{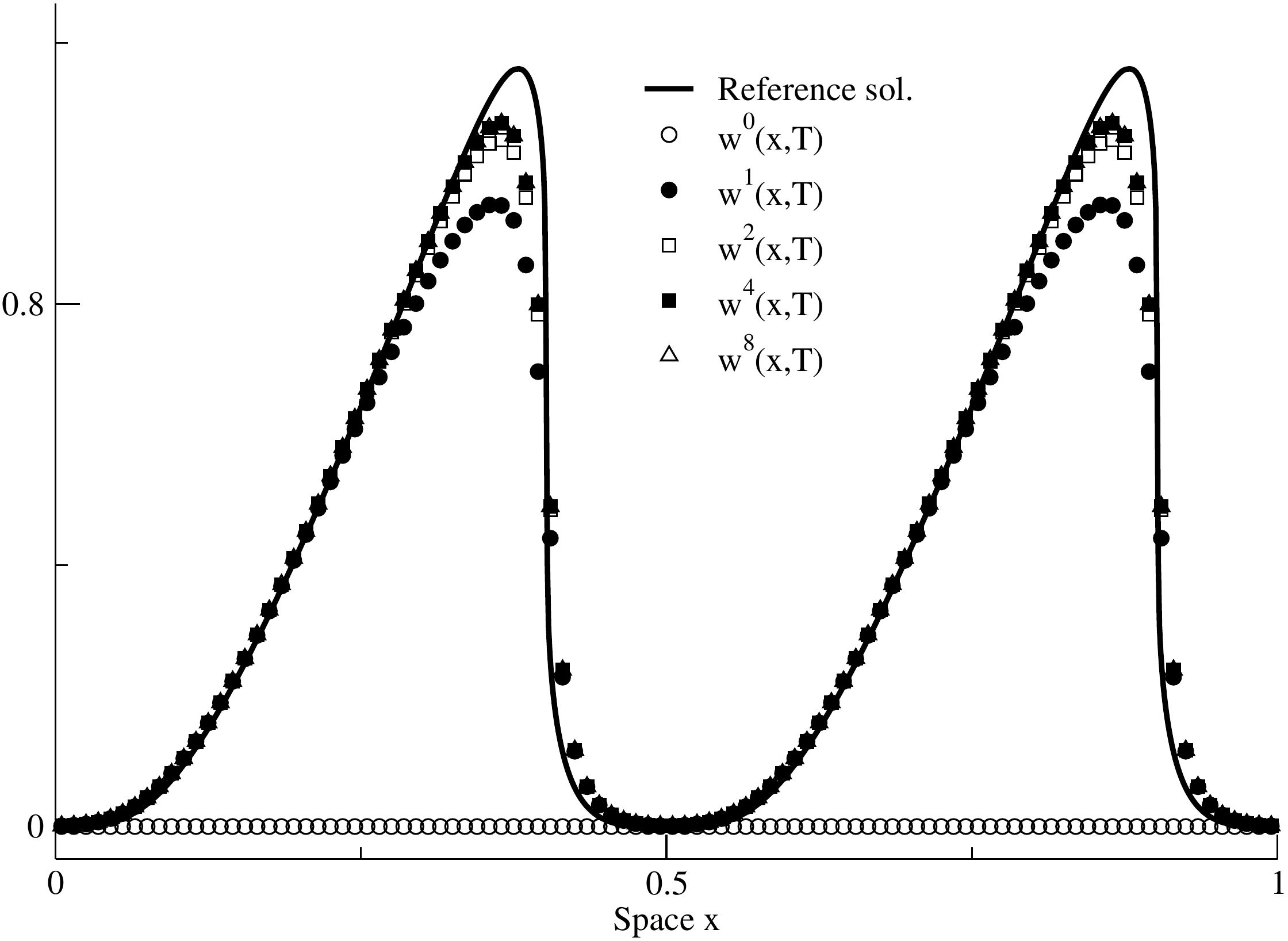}
\caption{Burger's equation. Comparison of solutions $w^0$, $w^1$, $w^2$, $w^4$ and $w^8$ at $T =0.12$, against the reference solution for  $100$ cells, $CFL = 0.5$ and $\alpha = 2.55$. }
\label{fig:burger-u-to-power-4}
\end{center}
\end{figure}

\begin{figure}
\begin{center}
\includegraphics[scale=0.5]{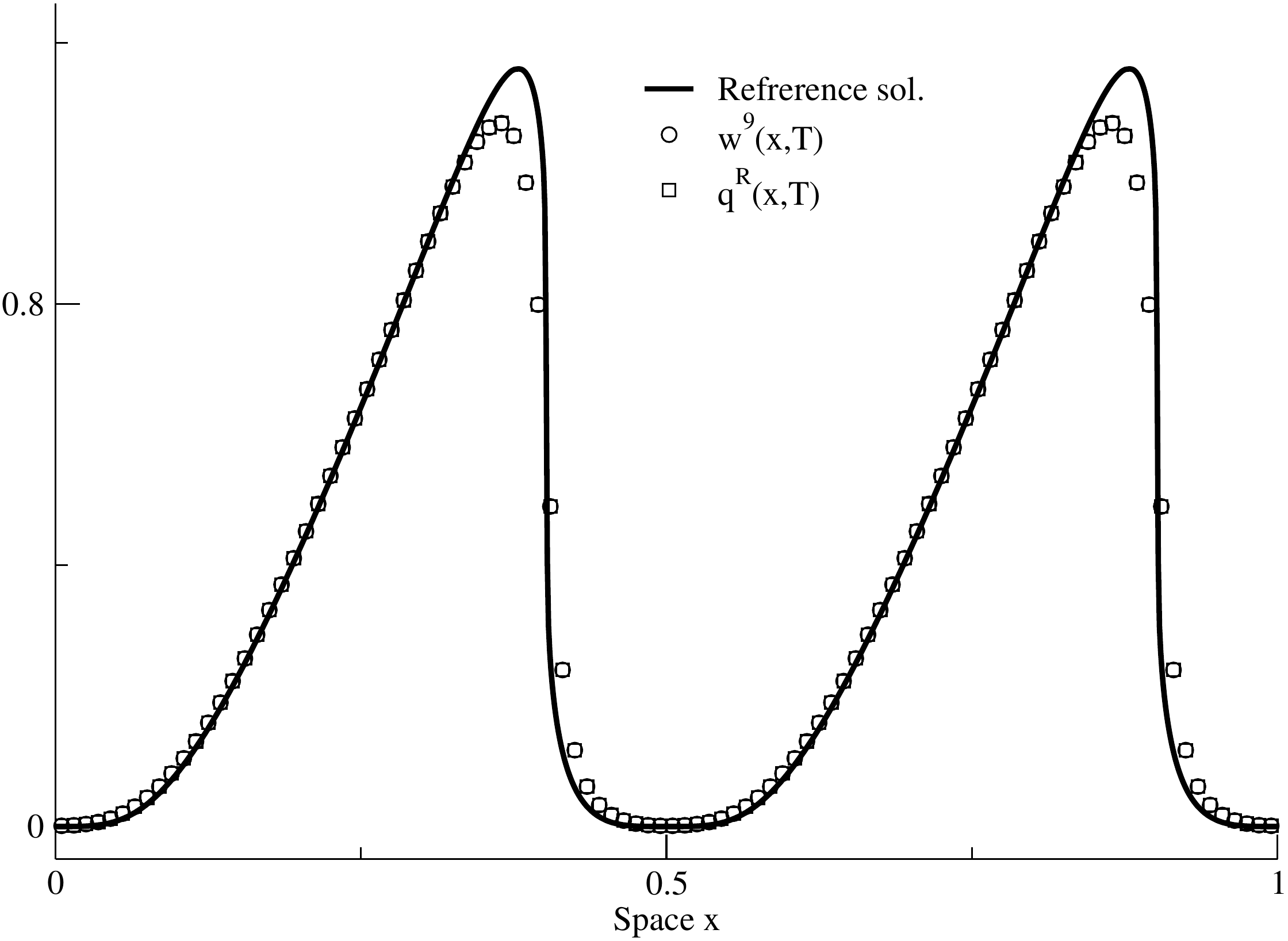}
\end{center}
\caption{Burger's equation. 
Comparison of converged solution $w^{9}$, approximated first order $q^R$ and the reference solution at $T =0.12$ for $100$ cells, $CFL = 0.5$, $\alpha = 2.55$ and $Tol = 10^{-7}$.}
\label{fig:burger-u-to-power-4-comparison}
\end{figure}
\section{The traffic flow model}
Let us consider the traffic flow equation with  non-linear source term
\begin{eqnarray}
\begin{array}{l}
\partial_t q + \partial_x( q \cdot u_{max} \cdot (1 - \frac{q}{ q_{max}} ) = r \cdot q^3\;, x \in [0, 1],\; t \in[0, T]\;, \\

q(x,0) = 0.2 \cdot \frac{(1 + c(x))}{2} + 2.2 \cdot \frac{(1 - c(x))}{2}\;.
 
\end{array}
\end{eqnarray}
where $c(x) = \frac{ (x-0.5)}{x^2 + \delta}$, with $\delta = 10^{-6}$, we implement transmissive boundary condition, model parameters $ r = 2$, $ u_{max} = 3$,$ q_{max} = 0.8$.  

Notice that this source satisfies also that $s(q)=r q^3$ is locally bounded in $\mathbb{R}$ and $s(0)=0$. The source terms is also locally Lipschitz continuous, furthermore, this is an infinitely continuously differentiable function in $\mathbb{R}$. Furthermore, since we apply transmissive boundary conditions we can take also a bounded support of $s$ taking $ supp(u_0)\subset supp(s) $. That is, this problem does satisfy the theorem \ref{theo-2}, hence the iterative procedure should provide the solution of this balance law.  The implementation is carried out with, $CFL = 0.5$, $\alpha = 2.0$, $100$ cells and $Tol = 10^{-7}$. Figure \ref{fig:traffic-flow} shows the results for $w^0$, $w^1$, $w^2$ and $w^4$ at $T = 0.02$, where the tendency to the reference solution is observed. Furthermore, Table \ref{table:traffic-Flow-u-to-power-4} shows the convergence of the iterative process for the traffic flow model.  
The second column shows the sequence $\{ \beta_n \}$. The third column shows the $L^1$ error at the output time between each solution $w^n$ and the reference solution. The fourth column shows the gaining factor, $\tau^n$. We observe that the converged solution $w^9$ has the same performance as the conventional first order approximation. Here, again the convergence is in terms of the tolerance $Tol$. 
Figure \ref{fig:traffic-flow-comparison} shows the comparison between the converged function $w^9$ and the approximate first order accurate solution $q^R$ obtained by the discretization of the original hyperbolic system. 

\begin{figure}
\begin{center}
\includegraphics[scale=0.5]{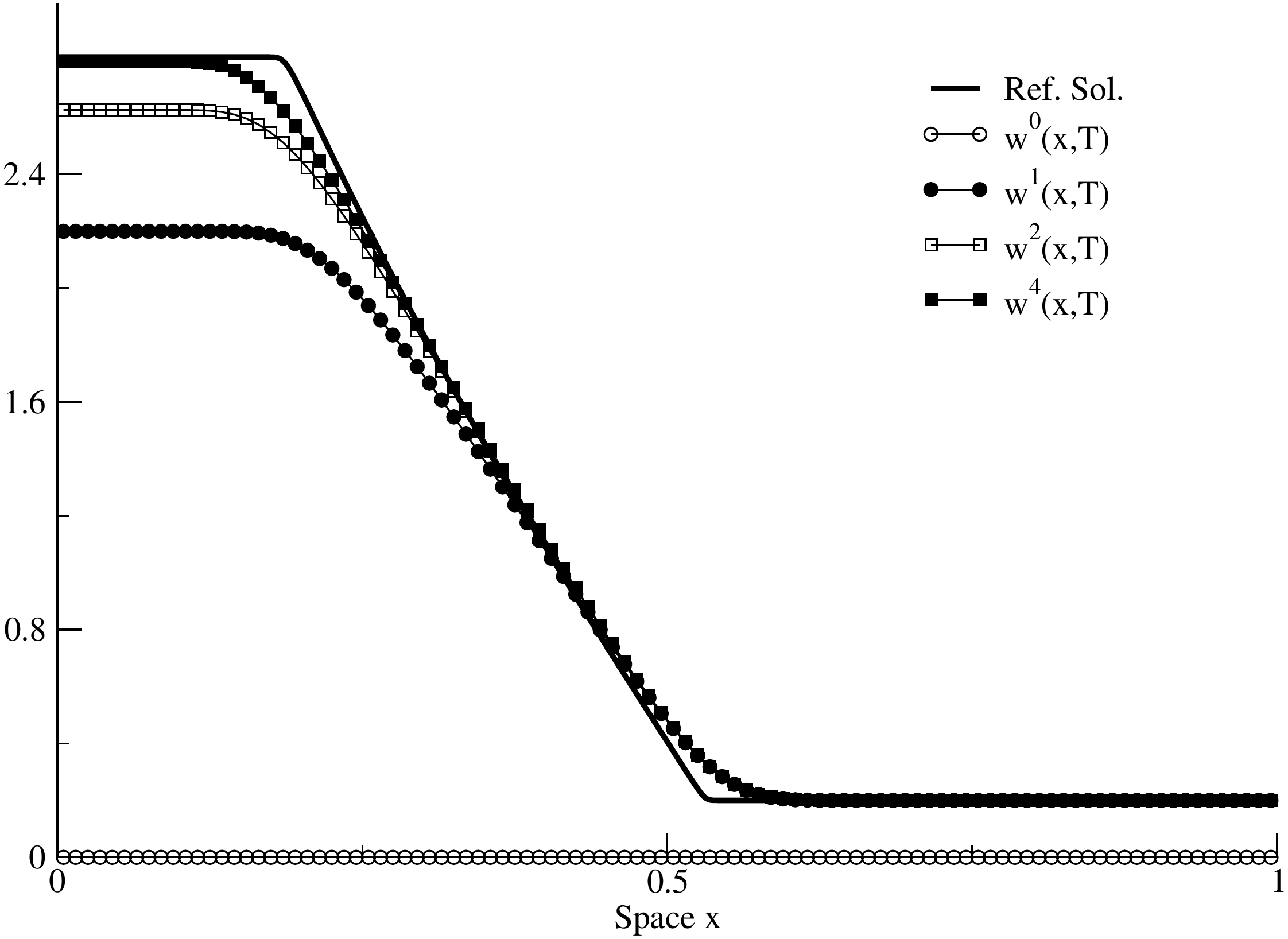}
\end{center}
\caption{Traffic flow equation. Comparison of solutions $w^0$, $w^1$, $w^2$ and $w^4$ at $T =0.02$, against the reference solution for $ r = 2$, $ u_{max} = 3$, $ q_{max} = 0.8$, $100$ cells, $CFL = 0.5$ and $\alpha = 2$.}\label{fig:traffic-flow}
\end{figure}

\begin{figure}
\begin{center}
\includegraphics[scale=0.5]{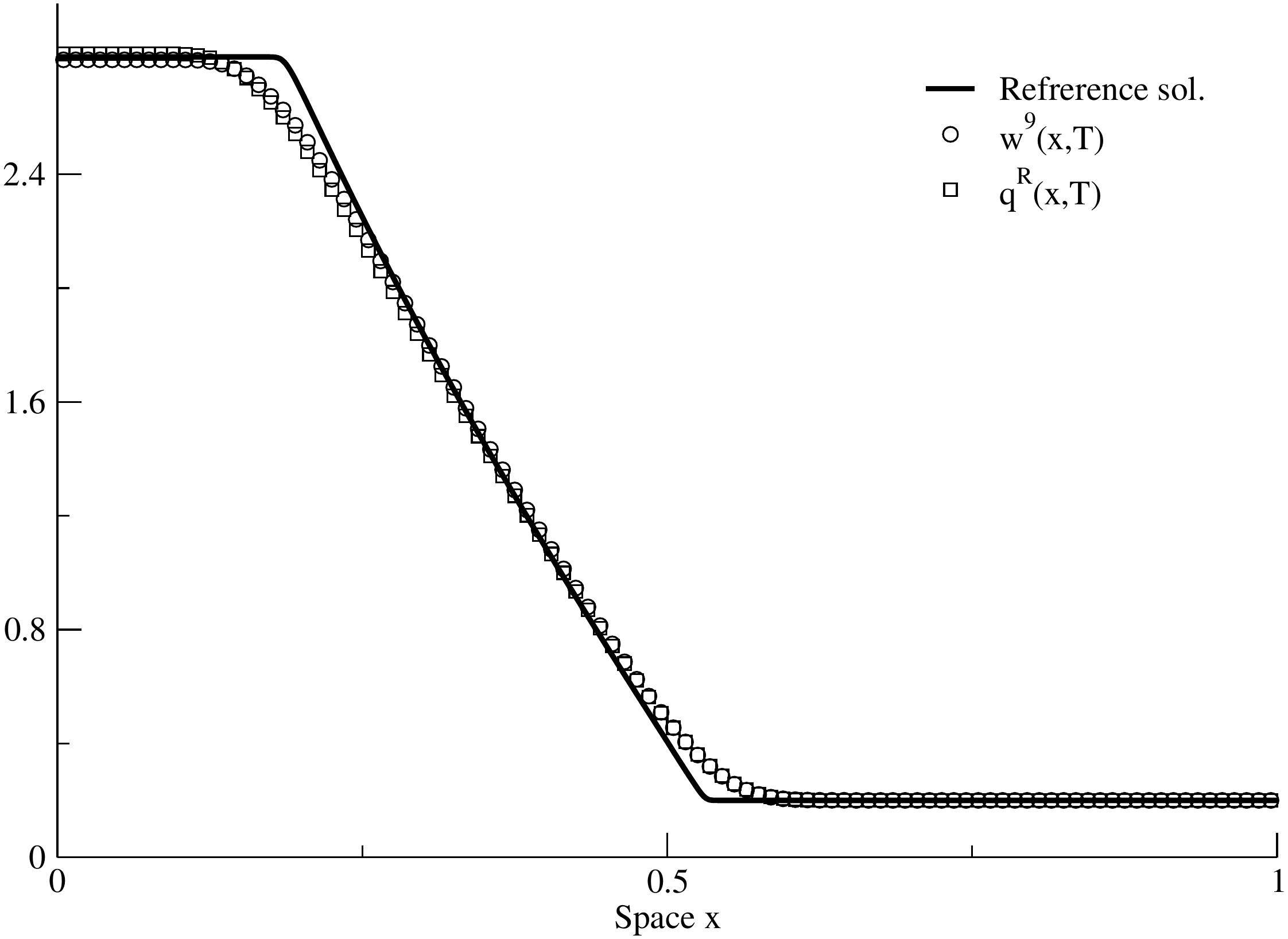}
\end{center}
\caption{Traffic flow equation. 
Comparison of converged solution $w^{9}$, approximated first order $q^R$ and the reference solution at $T =0.02$ for $ r = 2$, $ u_{max} = 3$, $ q_{max} = 0.8$, $100$ cells, $CFL = 0.5$, $\alpha = 2$ and $Tol = 10^{-7}$.}\label{fig:traffic-flow-comparison}
\end{figure}

\begin{table}
\begin{center}
\begin{tabular}{|c|c|c|c|}
\hline
$n$ & $\beta_n $ & $Err^n  $ & $\tau^n $ \\
\hline
\hline

 $ 1$  & $ 0.454545455$  & $0.177248859$  & $ 0.200584 $ \\
 $ 2$  & $ 0.380818913$  & $0.072762354$  & $ 0.488621 $ \\
 $ 3$  & $ 0.361848941$  & $0.043568247$  & $ 0.816036 $ \\
 $ 4$  & $ 0.357488961$  & $0.036910457$  & $ 0.963229 $ \\
 $ 5$  & $ 0.356703991$  & $0.035736574$  & $ 0.994870 $ \\
 $ 6$  & $ 0.356593379$  & $0.035573617$  & $ 0.999427 $ \\
 $ 7$  & $ 0.356580794$  & $0.035555224$  & $ 0.999944 $ \\
 $ 8$  & $ 0.356579603$  & $0.035553489$  & $ 0.999993 $ \\
 $ 9$  & $ 0.356579507$  & $0.035553349$  & $ 0.999997 $ \\

\hline
\end{tabular}
\end{center}
\caption{
Traffic flow equation. Second column: parameter $\beta_n$. Third column: Error of $w^k$ with respect to the exact solution. Fourth column: gaining in accuracy in the $k$ stage. Parameters: $ r = 2$, $ u_{max} = 3$, $ q_{max} = 0.8$, $100$ cells, $CFL = 0.5$, $\alpha = 2$ and $Tol = 10^{-7}$. }
\label{table:traffic-Flow-u-to-power-4}
\end{table}

\section{Conclusions}\label{sec:conclusions}
In this work, we have proved local existence and uniqueness in time, of a class of hyperbolic balance laws with non-linear source terms, satisfying that; i) $s(q) $ is locally bounded in $\mathbb{R}$ and $s(0)=0$; ii) The source terms is locally Lipschitz continuous and; $s \in C^2(\mathbb{R}) \cap H^2(\mathbb{R})$. The constructive proof is realizable into an iterative procedure for obtaining solutions for this class of balance laws, in which the convergence is guaranteed also by existence and uniqueness result. The iterative problem which in principle can be implemented by any approach, here has been solved by using a low-dissipation method in the framework of finite volume schemes. The approach generates a converged solution of first order of accuracy which coincides with the accuracy of the conventional implementation of a low-dissipation method. A high-order method generates a converged solution of high-order too, however, from experiments using a second order methods, not shown here, an increasing on the computational cost has been observed.
The demonstration of the convergence for the case of $L^\infty$ solutions with bounded variation and the exploration of this procedure for systems of hyperbolic balance laws are two important issues to be regarded in a future work.


\section*{Acknowledgements}

G.M thanks to the {\it National Research and Development Agency (Agencia Nacional de Investigación y Desarrollo, ANID)}, in the frame of the research project for Initiation in Research, number 11180926.

\section*{References}

\bibliographystyle{plain}
\bibliography{manuscript} 

\end{document}